\documentclass[preprint,12pt]{elsarticle}
\usepackage[all]{xy}
\usepackage{geometry}
\usepackage{graphicx}%
\usepackage{multirow}%
\usepackage{amsmath,amssymb,amsfonts}%
\usepackage{amsthm}%
\usepackage{mathrsfs}%
\usepackage{xcolor}%
\usepackage{textcomp}%
\usepackage{manyfoot}%
\usepackage{booktabs}%
\usepackage[utf8]{inputenc}

\theoremstyle{plain}
\newtheorem{theorem}{Theorem}[section]
\newtheorem{proposition}[theorem]{Proposition}
\newtheorem{corollary}[theorem]{Corollary}
\newtheorem{lemma}[theorem]{Lemma}
\newtheorem{definition}[theorem]{Definition}

\theoremstyle{definition}
\newtheorem{example}[theorem]{Example}
\newtheorem{remark}[theorem]{Remark}

\numberwithin{equation}{section}
\numberwithin{table}{section}

\DeclareMathOperator{\Tr}{Tr}
\def \F {{\mathbb F}}

\usepackage[hyphens]{url}
\usepackage{breakurl}
\usepackage[colorlinks=true,citecolor=blue,linkcolor=blue,urlcolor=blue,bookmarks,bookmarksopen,bookmarksdepth=2,backref=page,breaklinks]{hyperref}

\usepackage{bookmark}
\bookmarksetup{
	numbered, 
	open,
}

\renewcommand*{\backref}[1]{}
\renewcommand*{\backrefalt}[4]{%
	\ifcase #1 (Not cited.)%
	\or        (Cited on page~#2.)%
	\else      (Cited on pages~#2.)%
	\fi}

\journal{\phantom{}}
\begin{document}

\begin{frontmatter}

\title{Changing almost perfect nonlinear functions \\ on affine subspaces of small codimensions}

\author[inst1]{Hiroaki Taniguchi}\ead{taniguchi.hiroaki@yamato-u.ac.jp}

\affiliation[inst1]{organization={Yamato University, Department of Education}, addressline={2-5-1, Katayamacho}, city={Suita City}, postcode={564-0082},country={Japan}}

\author[inst2]{Alexandr Polujan}\ead{alexandr.polujan@gmail.com}
\author[inst2]{Alexander Pott}\ead{alexander.pott@ovgu.de}

\affiliation[inst2]{organization={Otto von Guericke University}, addressline={Universitätsplatz 2},city={Magdeburg},postcode={39106},country={Germany}}

\author[inst3]{Razi Arshad}\ead{r.arshad@kingston.ac.uk}
\affiliation[inst3]{organization={School of Computer Science and Mathematics}, addressline={Kingston University},city={London},postcode={KT12EE},country={UK}}

\begin{abstract}
In this article, we study algebraic decompositions and secondary constructions of almost perfect nonlinear (APN) functions. In many cases, we establish precise criteria which characterize when certain modifications of a given APN function yield new ones. Furthermore, we show that some of the newly constructed functions are extended-affine inequivalent to the original ones.
\end{abstract}

\begin{keyword}
APN function \sep Equivalence \sep Secondary construction \sep Exponential sum  
\MSC[2010] 06E30 \sep 11T06 \sep 94A60
\end{keyword}
\end{frontmatter}

\section{Introduction}

Let $\F_2^n$ be the $n$-dimensional vector space over the finite field with two elements $\F_2$. A mapping $F\colon\F_2^n\to\F_2^m$ is called an $(n,m)$-function. Almost perfect nonlinear (APN) functions are $(n,m)$-functions $F$ with optimal differential properties. Formally, it means that the cardinality of the set $\{x\in\F_2^n\colon F(x+a)+F(x)=b\}$, which is always even due to the parity of solutions $x$ and $x+a$, is less than or equal to $2$, for any nonzero $a\in {\Bbb F}_2^n$ and for any $b \in {\Bbb F}_2^m$. Originally introduced in the $(n,n)$-case, APN functions are valuable primitives in cryptographic applications, especially in the design of S-Boxes~\cite{carlet}. Due to their correspondence to certain optimal codes with minimum distance five~\cite{CCZ:1998} and the possibility of constructing from them difference sets~\cite{Dillon1999,DillonDobbertin2004} and large Sidon sets~\cite{CP_2024_Sidon}, they play a significant role in coding theory and, in general, discrete mathematics and combinatorics~\cite{Gologlu23Habil}. Recently, there has been an increasing interest in studying $(n,m)$-APN functions for the case $m>n$; for the recent works on the topic, we refer to~\cite{abbondati_Dproperty,taniguchi}. 

One of the central problems in the research on APN functions lies in discovering new constructions of these mappings. 
The known construction methods of such functions can be classified into primary and secondary. Primary constructions aim to provide new infinite families of APN functions ``from scratch'' (for the summary of the known infinite families, we refer to the recent article~\cite{Li_Kaleyski_2024}), while secondary constructions aim to construct new functions from the known ones. Some of the well-known secondary construction methods of APN functions (we briefly explain them further) include the switching construction (and its various generalizations), modifications by applying functions of a special shape, and isotopic shift construction, to name a few.

For a given $(n,n)$-APN function $F$, represented as a vector of coordinate single-output functions $F(x)=(f_1(x),f_2(x),\ldots,f_n(x))$, the switching construction aims at constructing a new APN function of the form $F'(x)=(f_1(x),f_2(x),\ldots,f'_n(x))$, i.e., by modifying a single coordinate function of a given APN function (w.l.o.g, the last one).  An advantage of this approach is that it preserves much of the information from the original function $F$, allowing room to demonstrate the APN property of a new function $F'$; see, e.g.,~\cite{BUDAGHYAN2009150,Edel_non_quadratic_APN:200959}. Note that switching~\cite{Edel_non_quadratic_APN:200959} and several related methods~\cite{Yu_QAM:2014,Kalgin2023}, which can be viewed as its generalizations, can be used both for generating numerous new examples through computer searches in fixed dimensions and constructing new infinite families of APN functions~\cite{BUDAGHYAN2009150}. In the latter case, one obtains APN functions of the form $G(x)=F(x)+Tr(H(x))$ defined on the finite field $\F_{2^n}$, which gives the motivation to study further generalization in polynomial form. For example, in~\cite{BCL_ITW_2009,Villa2019}, the authors studied the APN-ness of mappings~ $\F_{2^n}\ni x \mapsto L_1(x^3) + L_2(x^9)$ where $L_1$ and $L_2$ are linear. In~\cite{CharpinK17}, the authors investigated APN-ness of mappings of the form $G(x)=F(x)+\operatorname{Tr}(x)L(x)$, which corresponds to modifications by linear functions on hyperplanes. The studies~\cite{BCCCV_Isotopic_2020,Budaghyan2021} provide an analysis of the APN property of mappings defined as $G(x) = xL_1^{2^i}(x) + x^{2^i}L_2(x)$ on $\F_{2^n}$. This approach to constructing APN functions is called the (generalized) isotopic shift method.

While the previously discussed constructions focus on modifying a single APN function, another possible approach is to combine several given APN functions to build a new one. For example, the idea of constructing an APN function in $n+1$ variables from two APN functions in $n$ variables has been studied in~\cite{Kalgin2023} and later extended in~\cite{BeierleC23,BeierleLP22}. These results, in turn, stimulate further investigations of various algebraic decompositions of APN functions. For example, the connection between different APN functions was shown useful, when certain decompositions of $(n,n)$-APN functions helped to characterize the structure of quadratic APN functions with the maximum linearity~\cite{BeierleLP22}.

These findings motivate our further investigation of variations of the known secondary constructions of APN functions, their algebraic decompositions, and, in particular, modifications of APN functions on affine subspaces with small codimensions. This approach preserves much of the structure of the original APN function enabling the verification of the APN property for a modified function and, in many cases, allowing for an explicit characterization of this property.

\paragraph{Our contribution and the structure of the paper} The paper is organized as follows. In Subsection~\ref{sec: preliminaries} we give a necessary background on APN functions. In Section~\ref{sec 2 NEW}, we consider construction methods of $(n,m)$-APN functions. In Subsection~\ref{sec: 2}, we consider the switching construction for $(n,m)$-APN functions, and in Section~\ref{sec: 3}, we investigate the question of constructing an $(n,m)$-APN function $F$ from two $(n-1,m)$-APN functions $f$ and $g$ defined on complementary hyperplanes of $\F_2^n$. Notably, we uncover structural connections between such mappings and present a necessary and sufficient condition on the functions $f$ and $g$  to ensure that $F$ is APN (Theorem~\ref{thprop1}), along with a simplified condition for the case when both $f$ and $g$ are quadratic (Proposition~\ref{prop2}). In the remaining sections, we concentrate solely on the construction methods of $(n,n)$-APN functions. In Section~\ref{sec: 4}, we investigate how to construct a new quadratic APN function $G$ from a given quadratic APN function $F$ on $\F_2^n$ by adding to it a linear function defined on a hyperplane. Identifying $\F_2^n$ with the finite field $\F_{2^n}$, this question can be interpreted as finding APN functions on $\F_{2^n}$ of the form $G(x)=F(x)+\Tr(x)L(x)$, where $F$ is a quadratic APN function and $L$ is a linear mapping on $\F_{2^n}$. In Subsection~\ref{subsec: 4.1}, we provide a complete characterization of the APN property of such modifications (Theorem~\ref{th1}) and call such APN functions $G$ and $F$ ``hyperplane-equivalent'' or $H$-equivalent, for short  (Definition~\ref{def: H-equivalence}). Remarkably, we indicate that all extended-affine equivalence classes of quadratic APN functions on $\F_{2^6}$ have a representative of the form $G(x)=x^3+\Tr(x)L(x)$, which is $H$-equivalent to $F(x)=x^3$ (Theorem~\ref{rem: apns in dim 6}). Thus, we provide a new unifying approach to describe quadratic APN functions on $\F_{2^6}$. Motivated by these observations, in Subsection~\ref{subsec: 4.2}, we further investigate the APN property for quadratic APN functions that are $H$-equivalent to $x^3$. Particularly, we provide a necessary and sufficient condition for the APN-ness of mappings $\F_{2^n}\ni x \mapsto x^3+\Tr(x)L(x)$ using certain exponential sums (Theorem~\ref{th: Kloosterman sums iff}). 
Finally, in Section~\ref{sec: 5}, we present a necessary and sufficient condition (Theorem~\ref{th:Subspaces}) for constructing new APN functions from a given one by adding constant functions defined on complementary affine subspaces of codimension two. We indicate in Example~\ref{ex:subspace_n_8_1} that the proposed construction method yields extended-affine inequivalent functions. The paper is concluded in Section~\ref{sec: conclusion}.

\subsection{Preliminaries}\label{sec: preliminaries}
Let ${\Bbb F}_{2^n}$ be the finite field of $2^n$ elements. We sometimes identify ${\Bbb F}_{2^n}$ with ${\Bbb F}_2^n$ as an ${\Bbb F}_2$-vector space. We denote the sets ${\Bbb F}_{2^n}\backslash \{0\}$ by ${\Bbb F}_{2^n}^\times$ and ${\Bbb F}_2^n\backslash \{0\}$ by $({\Bbb F}_2^n)^\times$. For the finite fields $K\supset F$ of characteristic $2$, we denote the \textit{trace function} from $K$ to $F$ by ${\Tr}^K_F$. We denote ${\Tr}^K_{{\Bbb F}_2}$ by $\Tr$ and call it the \textit{absolute trace} of $K$. The mappings $\F_{2^n}\ni x \mapsto\sum\limits_{i=0}^{n-1}a_ix^{2^i}$, where all $a_i\in\F_{2^n}$, are called \textit{linearized polynomials}; these are linear mappings defined on the finite field $\F_{2^n}$.

A function $F\colon{\Bbb F}_2^n\rightarrow {\Bbb F}_2^m$ is called an \textit{almost perfect nonlinear (APN)} $(n,m)$-function  
if the cardinality of the set $\{x\in\F_2^n\colon F(x+a)+F(x)=b\}$ is less than or equal to $2$ for any nonzero $a\in {\Bbb F}_2^n$ and for any $b \in {\Bbb F}_2^m$. Sometimes we omit the specification of $n$ and $m$, when they are clear from the context. A function $F\colon{\Bbb F}_2^n\rightarrow {\Bbb F}_2^m$ is called \textit{quadratic} if the mapping ${\Bbb F}_2^m\times{\Bbb F}_2^m\ni (x,y) \mapsto F(x+y)+F(x)+F(y)+F(0)$ is ${\Bbb F}_{2}$-bilinear. 

On the set of all $(n,m)$-functions we introduce the following equivalence relations that preserve the APN property. The functions $F_1$ and $F_2$ from ${\Bbb F}_{2}^n$ to ${\Bbb F}_{2}^m$ are called \textit{Carlet-Charpin-Zinoviev equivalent} (\textit{CCZ-equivalent}, for short) if the graphs $G_{F_1}:=\{(x,F_1(x))\colon x\in {\Bbb F}_{2}^n\}$ and $G_{F_2}:=\{(x,F_2(x))\colon x\in {\Bbb F}_{2}^n\}$ in ${\Bbb F}_{2}^n\oplus {\Bbb F}_{2}^m$ are affine equivalent~\cite{CCZ:1998}, that is, if there exists an ${\Bbb F}_{2}$-linear isomorphism $l \in GL_2({\Bbb F}_{2}^n\oplus {\Bbb F}_{2}^m)$ and an element $v\in {\Bbb F}_{2}^n\oplus {\Bbb F}_{2}^m$ such that $l(G_{F_1})+v=G_{F_2}$. We say that $(n,m)$-functions $F$ and $F'$ are \textit{extended-affine equivalent (EA-equivalent)}, if there exist two affine permutations $A_1\colon \mathbb{F}_2^m \rightarrow \mathbb{F}_2^m, A_2\colon \mathbb{F}_2^n \rightarrow \mathbb{F}_2^n$ and an affine function $A_3\colon \mathbb{F}_2^n \rightarrow \mathbb{F}_2^m$ such that $A_1 \circ F \circ A_2 + A_3=F^{\prime}$. We note that in general, CCZ-equivalence is coarser than EA-equivalence, and that EA-equivalent functions are also CCZ-equivalent. However, in the case of quadratic APN functions, they coincide, as shown in~\cite{Yoshiara2012}.

Finally, we define the following invariant that will be used later to distinguish inequivalent functions. The \textit{$\Gamma$-rank} of a function $F\colon{\Bbb F}_{2}^n\rightarrow {\Bbb F}_{2}^m$ is the rank of the incidence matrix over ${\Bbb F}_{2}$ of the incidence structure $\{{\cal P}, {\cal B}, I\}$, where ${\cal P}={\Bbb F}_{2}^n\oplus {\Bbb F}_{2}^m$, ${\cal B}={\Bbb F}_{2}^n\oplus {\Bbb F}_{2}^m$ and $(a,b)I(u,v)$ for $(a,b)\in {\cal P}$ and $(u,v)\in {\cal B}$ if and only if $F(a+u)=b+v$. It is well-known that if two functions $F_1$ and $F_2$ from ${\Bbb F}_{2}^n$ to ${\Bbb F}_{2}^m$ are CCZ-equivalent, then they have the same $\Gamma$-rank, see~\cite{Edel_non_quadratic_APN:200959}.

\section{Constructing $(n,m)$-APN functions}\label{sec 2 NEW}
In this section, we study construction methods of $(n,m)$-APN functions using two different approaches. The first one is based on the well-known switching construction introduced for the case of $(n,n)$-APN functions, and the second one is based on the concatenation of two $(n-1,m)$-APN functions defined on complementary hyperplanes in $\F_2^n$.
\subsection{Revisiting the switching construction for $(n,m)$-APN functions}\label{sec: 2}
For a function $f\colon {\Bbb F}_2^n\rightarrow {\Bbb F}_2^m$, we define the mapping $B_f(x,t):=f(x+t)+f(x)+f(t)+f(0)$, for $x,t\in \F_2^n$. Recall the following definition and lemmas in \cite{taniguchi}.
\begin{definition}[Definition~1 of \cite{taniguchi}]For a function $f\colon {\Bbb F}_2^n\rightarrow {\Bbb F}_2^m$, let us define the subsets $D_f({\Bbb F}_2^n)$ and $D_f^*({\Bbb F}_2^n)$ of  ${\Bbb F}_2^m$ as
	\begin{equation*}
		\begin{split}
			D_f({\Bbb F}_2^n):=&\{B_f(x,t)+B_f(y,t)\mid x,y,t\in {\Bbb F}_2^n\} \mbox{ and},\\
			D_f^*({\Bbb F}_2^n):=&\{B_f(x,t)+B_f(y,t)\mid x,y,t\in {\Bbb F}_2^n \text{ with } x\neq y, t\neq 0 \text{ and }x\neq y+t \}
		\end{split}
	\end{equation*}
\end{definition}
We denote by $\pi \circ f$ the composition mapping defined by $(\pi\circ f)(x):=\pi(f(x))$. 
\begin{lemma}[Lemma~2 of \cite{taniguchi}]\label{prop0}
Let $f\colon {\Bbb F}_2^n\rightarrow {\Bbb F}_2^m$ be an APN function.  Let $l<m$ be a positive integer and let $\pi\colon{\Bbb F}_2^m\rightarrow {\Bbb F}_2^{l}$ be an ${\Bbb F}_2$-linear surjection. 
Then, $\pi\circ f$ is an APN function if and only if $D_f^*({\Bbb F}_2^n)\cap \ker(\pi)=\emptyset$.
\end{lemma}

Recall that a function $f\colon{\Bbb F}_2^n\rightarrow {\Bbb F}_2^m$ is called a \textit{differentially $4$-uniform} function if $\vert\{x\in\F_2^n\colon f(x+a)+f(x)=b\}\vert\le 4$ for any nonzero $a\in {\Bbb F}_2^n$ and for any $b \in {\Bbb F}_2^m$.
\begin{lemma}[Lemma~3 of \cite{taniguchi}]\label{lemmax2}
Let $f\colon {\Bbb F}_2^n\rightarrow {\Bbb F}_2^m$ be an APN function.
Let $\pi\colon {\Bbb F}_2^m\rightarrow {\Bbb F}_2^{m-1}$ be an ${\Bbb F}_2$-linear surjection such that $D_f^*({\Bbb F}_2^n)\cap \ker(\pi)\neq \emptyset$. 
Then, $\pi\circ f\colon {\Bbb F}_2^n\rightarrow {\Bbb F}_2^{m-1}$ is a differentially $4$-uniform function.
\end{lemma}
The following proposition is a slight generalization of the switching construction (Theorem~3 of \cite{Edel_non_quadratic_APN:200959}).
\begin{proposition}\label{prop3}
Let $f\colon {\Bbb F}_2^n\rightarrow {\Bbb F}_2^m$ be a function and $g\colon {\Bbb F}_2^n\rightarrow {\Bbb F}_2$ a Boolean function such that $F\colon {\Bbb F}_2^n\ni x \mapsto (f(x),g(x))\in {\Bbb F}_2^m\oplus {\Bbb F}_2$ is an APN $(n,m+1)$-function.
Let $u$ be a nonzero element of ${\Bbb F}_2^m$.
Then, $f(x)+ug(x)$ is an APN $(n,m)$-function if and only if 
$f(x+t)+f(x)+f(y+t)+f(y)=u$ implies $g(x+t)+g(x)+g(y+t)+g(y)=0$.\end{proposition}
\begin{proof}  
Let $\pi\colon {\Bbb F}_2^m\oplus {\Bbb F}_2\ni (x,y)\mapsto x+uy\in {\Bbb F}_2^m$ be an ${\Bbb F}_2$-linear surjection with $\ker(\pi)=\langle (u,1)\rangle$.
Then, $(\pi\circ F)(x)=f(x)+ug(x)$.
By Lemma~\ref{prop0}, $\pi\circ F$ is an APN function if and only if $(u,1)\notin D_F^*({\Bbb F}_2^n)$, which is equivalent to the condition ``$B_f(x,t)+B_f(y,t)=f(x+t)+f(x)+f(y+t)+f(y)=u$ implies $B_g(x,t)+B_g(y,t)=g(x+t)+g(x)+g(y+t)+g(y)=0$", because $B_F(x,t)+B_F(y,t)=(B_f(x,t)+B_f(y,t), B_g(x,t)+B_g(y,t))$.
\end{proof}

\begin{remark}
If $f$ in Proposition~\ref{prop3} is not an APN function, then $f$ is a differentially $4$-uniform function by Lemma~\ref{lemmax2}.
Let $f$ be a differentially $4$-uniform function. Thus, $F\colon {\Bbb F}_2^n\ni x \mapsto (f(x),g(x))\in {\Bbb F}_2^m\oplus {\Bbb F}_2$ is an APN $(n,m+1)$-function if and only if the following conditions are satisfied:
for $x\neq y, x\neq y+t$ and $t\neq 0$, $f(x+t)+f(x)+f(y+t)+f(y)=0$ implies $g(x+t)+g(x)+g(y+t)+g(y)=1$.
\end{remark}

\begin{lemma}[See \cite{nyberg}]\label{lemma2}
Let $n=2k$ be an even integer with $k>1$.
Then, for any $a \in {\Bbb F}_{2^n}^\times$ and $b \in {\Bbb F}_{2^n}$, the following statements hold.
\begin{enumerate}
\item $x^{-1}+(x+a)^{-1}=b$ has no root in ${\Bbb F}_{2^n}$ if and only if $\Tr(\frac{1}{ab})=1$.
\item $x^{-1}+(x+a)^{-1}=b$ has $2$ root in ${\Bbb F}_{2^n}$ if and only if $ab\neq 1$ and $\Tr(\frac{1}{ab})=0$.
\item $x^{-1}+(x+a)^{-1}=b$ has $4$ root in ${\Bbb F}_{2^n}$ if and only if $b=a^{-1}$.
Furthermore, when $b=a^{-1}$ the $4$ root of the above equation in ${\Bbb F}_{2^n}$ are $\{0, a, \omega a, \omega^2 a\}$, where $\omega \in {\Bbb F}_{2^2}\backslash {\Bbb F}_{2}$.
\end{enumerate} 
\end{lemma}

Now, we give an example of a differentially $4$-uniform function $f$ on ${\Bbb F}_{2^n}$ and a Boolean function $g\colon {\Bbb F}_{2^n}\rightarrow {\Bbb F}_{2}$ such that $F(x)=(f(x),g(x))$ is an $(n,n+1)$-APN function. 
\begin{example}
Let $n=2k$ be an even integer with $k>1$. 
Then, the inverse function $f(x)=x^{2^n-2}$ is a differentially $4$-uniform function on ${\Bbb F}_{2^n}$.  
By Lemma~\ref{lemma2}, we see that $f(x+a)+f(x)+f(y+a)+f(y)=0$ with $x\neq y$, $x\neq y+a$ and $a\neq 0$ implies $\{x, x+a,y,y+a\}=\{0,a,\omega a, \omega^2 a\}$.
Hence, for a Boolean function $g\colon {\Bbb F}_{2^n}\rightarrow {\Bbb F}_{2}$, we see that $F\colon {\Bbb F}_{2^n}\ni x \rightarrow (f(x),g(x))\in {\Bbb F}_{2^n}\oplus {\Bbb F}_2$ is an $(n,n+1)$-APN function if and only if 
 $g(0)+g(a)+g(\omega a)+g(\omega^2 a)=1$ for any $a\in {\Bbb F}_{2^n}^\times$.

Let $A:=\{x^3\mid x\in {\Bbb F}_{2^n}^\times\}$.
Then, ${\Bbb F}_{2^n}^\times=A\cup \omega A\cup \omega^2 A$, where $\omega A=\{\omega a\mid a\in A\}$ and $\omega^2 A=\{\omega^2 a\mid a\in A\}$.
Let us define $g$ by $\{g(a), g(\omega a), g(\omega^2 a)\}=\{0,0,1\}$ or $\{1,1,1\}$ for $a \in A$, and $g(0)=0$. 
Then, we see that $g(0)+g(a)+g(\omega a)+g(\omega^2 a)=1$ for any $a\in {\Bbb F}_{2^n}^\times$ by definition. Hence, $F\colon {\Bbb F}_{2^n}\ni x \rightarrow (f(x),g(x))\in {\Bbb F}_{2^n}\oplus {\Bbb F}_2$ is an $(n,n+1)$-APN function.
\end{example}

In the following proposition, we show that every APN function on ${\Bbb F}_{2^n}$ can be obtained from a differentially $4$-uniform function using the switching construction in Proposition~\ref{prop3}. While most of the secondary constructions of new APN functions are based on constructing them from the known APN functions, this result provides a strategy for finding new APN functions from suitably chosen differentially 4-uniform functions. The proof of this statement is not difficult to see, nevertheless, we give it for the reader's convenience.
\begin{proposition}\label{prop: switching}
Let $f$ be an APN function on ${\Bbb F}_{2^n}$.
Then, there exist a differentially $4$-uniform function $f_1$ on ${\Bbb F}_{2^n}$, a Boolean function $g\colon {\Bbb F}_{2^n}\rightarrow {\Bbb F}_{2}$ and an element $u \in {\Bbb F}_{2^n}^\times$, such that $f(x)=f_1(x)+ug(x)$ holds for all $x\in {\Bbb F}_{2^n}$.   
\end{proposition}

\begin{proof}
Let $f$ be any APN function on ${\Bbb F}_{2^n}$ and $g\colon {\Bbb F}_{2^n}\rightarrow {\Bbb F}_{2}$ be a Boolean function.
Then, $F(x):=(f(x),g(x))$ is an APN $(n,n+1)$-function.
Moreover, we assume that $g$ satisfies, for some $x,y,t\in {\Bbb F}_{2^n}$, $B_{g}(x,t)+B_{g}(y,t)=g(x)+g(x+t)+g(y)+g(y+t)=1$. (Obviously, we have that $x\neq y, t\neq 0 \text{ and }x\neq y+t$ for this $x,y,t$.) 
We note that $D_{f}^*({\Bbb F}_2^n)={\Bbb F}_{2^n}^\times$ since $f$ is an $(n,n)$-APN function. (It is the ``Dillon's observation''; see p.~381 of \cite{carlet}).
Since  $D_{F}^*({\Bbb F}_{2^n})=\{(B_{f}(x,t)+B_{f}(y,t),B_{g}(x,t)+B_{g}(y,t)) \mid x,y,t\in {\Bbb F}_2^n \text{ with } x\neq y, t\neq 0 \text{ and }x\neq y+t \}$, there exists $(u,1)\in D_{F}^*({\Bbb F}_{2^n})$ for some $u\in {\Bbb F}_{2^n}^\times$. 
Let $\pi\colon {\Bbb F}_{2^n}\oplus {\Bbb F}_{2}\ni (x,y)\mapsto x+uy\in {\Bbb F}_{2^n}$ be a surjective linear mapping with the kernel $\langle (u,1)\rangle$. Then, $f_1(x):=\pi(F(x))=f(x)+ug(x)$ is a differentially $4$-uniform function on ${\Bbb F}_{2^n}$ by Lemma~\ref{lemmax2}. Thus, the APN function $f$ can be expressed as $f(x)=f_1(x)+ug(x)$ using a differentially $4$-uniform function $f_1$, a Boolean function $g$ and an element $u \in {\Bbb F}_{2^n}^\times$.
   \end{proof}

\begin{remark}
	The discussion in~\cite[p. 407]{carlet} explains that for an APN function $f_1$, a Boolean function $g$ and an element $u\in \F_{2^n}^\times$, the function $f:=f_1+ug$ is differentially $4$-uniform in general, and APN in rare cases. Note that this reasoning can also be extended to the case where $f_1$ is differentially $\delta$-uniform. Proposition~\ref{prop: switching}, on the other hand, states that for any APN function $f$, there exists a differentially $4$-uniform function $f_1$, a Boolean function $g$ and an element $u\in \F_{2^n}^\times$, such that $f$ can be expressed as $f:=f_1+ug$.
    This latter expression, although trivial, expresses our point of view: we want to \textit{firstly} construct a differentially $4$ uniform function and then \textit{secondly} modify this to transform it into an APN function.
\end{remark}

\subsection{Constructing $(n,m)$-APN functions from $(n-1,m)$-APN functions}\label{sec: 3}

Let $n\ge 3$ and $m\ge n$. Let $F$ be an APN function from ${\Bbb F}_2^{n}$ to ${\Bbb F}_2^m$.
We identify ${\Bbb F}_2^{n-1}$ as an $(n-1)$-dimensional subspace (a hyperplane) of ${\Bbb F}_2^n$. Let $e_0\in {\Bbb F}_2^n$ with $e_0\not\in {\Bbb F}_2^{n-1}$.
Then, the complement of ${\Bbb F}_2^{n-1}$ can be expressed as ${\Bbb F}_2^{n-1}+e_0:=\{x+e_0\colon x\in {\Bbb F}_2^{n-1}\}$. Moreover, we have ${\Bbb F}_2^n={\Bbb F}_2^{n-1}\cup ({\Bbb F}_2^{n-1}+e_0)$.  Let us define $f(x):=F(x)$ for $x\in {\Bbb F}_2^{n-1}$, and $g(x):=F(x+e_0)$ for $x\in {\Bbb F}_2^{n-1}$.
Then, the functions $f,g$ from ${\Bbb F}_2^{n-1}$ to ${\Bbb F}_2^m$ are APN. In this section, we investigate under which conditions one can ``concatenate'' two APN functions $f,g$ from ${\Bbb F}_2^{n-1}$ to ${\Bbb F}_2^m$ in order to get an APN mapping $F$ from ${\Bbb F}_2^n={\Bbb F}_2^{n-1}\cup ({\Bbb F}_2^{n-1}+e_0)$ to ${\Bbb F}_2^m$ defined by $F(x)=f(x)$ and $F(x+e_0)=g(x)$, for $x\in {\Bbb F}_2^{n-1}$. 
In the following theorem, we provide a necessary and sufficient condition for $F\colon {\Bbb F}_2^n \to {\Bbb F}_2^m$ to be APN under this construction.

\begin{theorem}\label{thprop1}
Let $f,g$ be functions from ${\Bbb F}_2^{n-1}$ to ${\Bbb F}_2^m$.
Let $e_0\in {\Bbb F}_2^{n}\backslash {\Bbb F}_2^{n-1}$.
Let us define $F$ from ${\Bbb F}_2^{n}$ to ${\Bbb F}_2^m$ by $F(x):=f(x)$ and $F(x+e_0):=g(x)$ for $x\in {\Bbb F}_2^{n}$.
Then, $F$ is an APN function if and only if the following conditions hold:
\begin{itemize}
\item[$(1)$] $f$ and $g$ are APN functions from ${\Bbb F}_2^{n-1}$ to ${\Bbb F}_2^m$,
\item[$(2)$] $f(x+a)+f(x)+ g(y+a)+g(y)\ne 0$ holds for any $x,y\in {\Bbb F}_2^{n-1}$ and for any nonzero $a\in {\Bbb F}_2^{n-1}$. 
\end{itemize}
\end{theorem}
\begin{proof}
Recall that $F$ is an APN function if and only if, for any nonzero $A\in {\Bbb F}_2^n$ and for $X,Y\in {\Bbb F}_2^n$,
 $F(X+A)+F(X)=F(Y+A)+F(Y)$ implies $X=Y$ or $X=Y+A$.
 Note that condition (2) means: If $f(x)+f(y)+g(u)+g(v)\ne 0$ if $u,v,x,y$ are four different elements in $\F_2^{n-1}$ with $u+v+x+y=0$.

Firstly, assume that $F$ is an APN function. We will show that $f$ and $g$ must satisfy the conditions $(1)$ and $(2)$.

Let $A=a\in ({\Bbb F}_2^{n-1})^\times$. For any $Y=y \in {\Bbb F}_2^{n-1}$, 
we must have $X=y\in {\Bbb F}_2^{n-1}$ or $X=y+a \in {\Bbb F}_2^{n-1}$ from $F(X+a)+F(X)=F(y+a)+F(y)$. Since $X \in {\Bbb F}_2^{n-1}$, we have $f(X+a)+f(X)=f(y+a)+f(y)$ from $F(X+a)+F(X)=F(y+a)+F(y)$. Thus, $f$ must be an APN function.
Next, for any $Y=y+e_0$ with $y \in {\Bbb F}_2^{n-1}$ we must have $X=y+e_0$ or $X=y+a+e_0$ from $F(X+a)+F(X)=F(y+e_0+a)+F(y+e_0)$.
Since $X=x+e_0$ for some $x\in {\Bbb F}_2^{n-1}$, we have $g(x+a)+g(x)=g(y+a)+g(y)$ from $F(X+a)+F(X)=F(y+e_0+a)+F(y+e_0)$. Hence, $g$ must be an APN function.
Thus, the condition $(1)$ must be satisfied.

Let $A=a\in ({\Bbb F}_2^{n-1})^\times$. For any $Y=y \in {\Bbb F}_2^{n-1}$, since $X=y$ or $X=y+a$, $F(X+a)+F(X)=F(y+a)+F(y)$ does not have a solution $X=x+e_0$ for $x\in {\Bbb F}_2^{n-1}$. Thus, $F(x+e_0+a)+F(x+e_0)\neq F(y+a)+F(y)$ holds for any $x,y \in {\Bbb F}_2^{n-1}$, therefore we must have $g(x+a)+g(x)\neq f(y+a)+f(y)$ for any $x,y \in {\Bbb F}_2^{n-1}$.
Thus, the condition $(2)$ must be satisfied.


Conversely, let us assume that the conditions $(1)$ and $(2)$ hold.
Assume $F(X+A)+F(X)=F(Y+A)+F(Y)$ with $A\neq 0$. We will prove that $X=Y$ or $X=Y+A$.
We divide this into the following four cases:
\begin{itemize}
    \item[(i)] $A=a\in ({\Bbb F}_2^{n-1})^\times$ and $Y=y \in {\Bbb F}_2^{n-1}$,
    \item[(ii)] $A=a\in ({\Bbb F}_2^{n-1})^\times$ and $Y=y+e_0$ with $y \in {\Bbb F}_2^{n-1}$,
    \item[(iii)] $A=a+e_0$ with $a\in {\Bbb F}_2^{n-1}$ and $Y=y$ with $y \in {\Bbb F}_2^{n-1}$,
    \item[(iv)] $A=a+e_0$ with $a\in {\Bbb F}_2^{n-1}$ and $Y=y+e_0$ with $y \in {\Bbb F}_2^{n-1}$.
\end{itemize}

Firstly, let us consider the case (i).  
If $X=x\in {\Bbb F}_2^{n-1}$, then we have $f(x+a)+f(x)=f(y+a)+f(y)$, hence $x=y$ or $x=y+a$ by $(1)$. Let $X=x+e_0$ with $x\in {\Bbb F}_2^{n-1}$, then we have $g(x+a)+g(x)=f(y+a)+f(y)$ which has no solution by $(2)$.
Therefore, $X=Y$ or $X=Y+A$ in case (i).

Next, we consider the case (ii).  
Assume $X=x\in {\Bbb F}_2^{n-1}$. Then, we have $f(x+a)+f(x)=g(y+a)+g(y)$ which has no solution by $(2)$.
If $X=x+e_0$ with $x\in {\Bbb F}_2^{n-1}$, then we have $g(x+a)+g(x)=g(y+a)+g(y)$, hence $x+e_0=y+e_0$ or $x+e_0=y+e_0+a$ by $(1)$.
Thus, we have $X=Y$ or $X=Y+A$ in case (ii).

Let us consider the case (iii).
If $X=x\in {\Bbb F}_2^{n-1}$, then we have $g(x+a)+f(x)+g(y+a)+f(y)=0$, hence
$x=y$.
If $X=x+e_0$ with $x\in {\Bbb F}_2^{n-1}$, then we have, again,  $f(x+a)+g(x)+g(y+a)+f(y)=0$. For the same reason as above, we have $x+e_0=y+(a+e_0)$. 
Therefore we have $x=y$ or  $x+e_0=y+(a+e_0)$. Thus, we have $X=Y$ or $X=Y+A$ in case (iii).

Lastly, we consider the case (iv).
If $X=x\in {\Bbb F}_2^{n-1}$, then we have $g(x+a)+f(x)+f(y+a)+g(y)=0$, hence
$x=(y+e_0)+(a+e_0)$.
If $X=x+e_0$ with $x\in {\Bbb F}_2^{n-1}$, then we have $f(x+a)+g(x)+f(y+a)+g(y)=0$. 
For the same reason as above, we have $x+e_0=y+e_0$.
Thus, we also have $X=Y$ or $X=Y+A$ in case (iv).

Hence, $F$ must be an APN function under the conditions $(1)$ and $(2)$.
\end{proof}

Now, let $F$ be a \textit{quadratic} APN function from ${\Bbb F}_2^{n}$ to ${\Bbb F}_2^m$. As before, for an element $e_0\not\in {\Bbb F}_2^{n-1}$, we define $f(x):=F(x)$ for $x\in {\Bbb F}_2^{n-1}$, and $g(x):=F(x+e_0)$ for $x\in {\Bbb F}_2^{n-1}$. In this case, we get two quadratic APN functions $f,g$ from ${\Bbb F}_2^{n-1}$ to ${\Bbb F}_2^m$, where $g(x)=f(x)+L(x)+c$ for an element $c:=F(e_0)+F(0)$ and a linear mapping $L(x):=B_F(x,e_0)$ with $B_F(x,e_0):=F(x+e_0)+F(x)+F(e_0)+F(0)$. In the following statement, we provide a necessary and sufficient condition that characterizes the APN-ness of the quadratic function  $F$ defined by $F(x):=f(x)$ and $F(x+e_0):=g(x)$, for $x\in {\Bbb F}_2^{n-1}$.

\begin{proposition}\label{prop2}
Let $f,g$ be quadratic APN functions from ${\Bbb F}_2^{n-1}$ to ${\Bbb F}_2^m$ such that $g(x)=f(x)+L(x)+c$, where $L$ is an ${\Bbb F}_2$-linear mapping from ${\Bbb F}_2^{n-1}$ to ${\Bbb F}_2^m$ and $c\in {\Bbb F}_2^m$.
Let $F$ be a function from ${\Bbb F}_2^n$ to ${\Bbb F}_2^{m}$ defined by $F(x):=f(x)$ and $F(x+e_0):=f(x)+L(x)+c$ for some fixed $e_0\in {\Bbb F}_2^{n} \backslash {\Bbb F}_2^{n-1}$, for $x\in {\Bbb F}_2^{n-1}$.
Then, $F$ is a quadratic APN function if and only if
\begin{equation}\label{eq:POTT}
{\Bbb F}_2^{n-1}\ni x\mapsto L(x)+B_f(x,A)\in {\Bbb F}_2^{m}
\end{equation}
are one-to-one mappings, for any $A\in {\Bbb F}_2^{n-1}$. 
\end{proposition} 
\begin{proof}
We will check the conditions $(1)$ and $(2)$  in Theorem~\ref{thprop1}.


\noindent Since $f$ and $g=f+L+c$ are APN functions, the condition $(1)$ is satisfied.

The condition $(2)$ is equivalent to $f(x+a)+f(x)\neq 
g(y+a)+g(y)=f(y+a)+f(y)+L(a)$ for any $x,y\in {\Bbb F}_2^{n-1}$ if $a\neq 0$, that is, $L(a)+(f(x+a)+f(x))+(f(y+a)+f(y))\neq 0$ for any $x,y\in {\Bbb F}_2^{n-1}$ if $a\neq 0$, which means $L(a)+B_f(a,x+y)\neq 0$ if $a\neq0$, $a\in {\Bbb F}_2^{n-1}$. Note that $f(x+a)+f(x)+f(y+a)+f(y)=B_f(a,x)+B_f(a,y)=B_f(a,x+y)$, 
where the second equality holds since $f$ is quadratic.
We put $A:=x+y$ and think of $a$ as the variable $x$ 
in the linear equation defined by Eq.~\eqref{eq:POTT}.  Using the fact that Eq.~\eqref{eq:POTT} describes a linear function, we see that $L(a)+B_f(a,x+y)\neq 0$ if $a\neq0$
is exactly the condition in Eq.~\eqref{eq:POTT}.
\end{proof}


\section{Modifying quadratic APN functions on hyperplanes}\label{sec: 4}
In the previous section, we considered a method for constructing APN functions by concatenating \textit{two} APN functions defined on complementary hyperplanes. An alternative approach to constructing new APN functions involves modifying a \textit{single} given APN function on a hyperplane. In this case, the APN property is characterized as follows:

\begin{corollary}\cite[Corollary 1]{CharpinK17}\label{cor: Pascale and Gohar}
	Let $T_0=\{x \colon \operatorname{Tr}(x)=0\}$ and $G\colon \mathbb{F}_{2^n} \rightarrow \mathbb{F}_{2^n}$ be defined by 
	$$
	G(x)=F(x)+\operatorname{Tr}(x)L(x) ,
	$$
	where $F$ and $L$ are arbitrary mappings on $\mathbb{F}_{2^n}$ (where $L$ is not necessarily linear). Then, for every $a \in T_0$ we have
	$$
	D_a G(x)=(F(x)+F(x+a))+\operatorname{Tr}(x)(L(x)+L(x+a)).
	$$
	Suppose that the mappings $F(x)$ and $F'(x):=F(x)+L(x)$ are APN. Set $I_a=\operatorname{Im}\left(D_a F\right) \cap \operatorname{Im}\left(D_a F'\right)$. Then, $G$ is APN if and only if for any nonzero $a \in T_0$ and every $b \in I_a$ there is $x \in \mathbb{F}_{2^n}$ with
	$$
	D_a F(x)=b \text { and } \operatorname{Tr}(x)=1 \text {, or } D_a F'(x)=b \text { and } \operatorname{Tr}(x)=0 .
	$$
\end{corollary}

In this section, we consider in detail a particular case when $F$ is \textit{quadratic APN} and $L$ is \textit{linear}.

\subsection{A characterization for the quadratic case and the notion of H-equivalence}\label{subsec: 4.1}

First, we derive a necessary and sufficient condition for the APN-ness of the functions $F(x)+\Tr(x)L(x)$ on $\F_{2^n}$, where $F$ is a quadratic APN function and $L$ is a linear mapping on $\F_{2^n}$.

\begin{theorem}\label{th1}
Let $F$ be a quadratic APN function on ${\Bbb F}_{2^n}$ and $L$ an ${\Bbb F}_{2}$-linear mapping on ${\Bbb F}_{2^n}$. Let $e_0\in {\Bbb F}_{2^n}$ with $\Tr(e_0)=1$.  
Then, the mapping
$$G(x)=F(x)+\Tr(x)L(x),\quad\mbox{for}\quad x\in\F_{2^n},$$
is a quadratic APN function on ${\Bbb F}_{2^n}$ if and only if $L_a\colon T_0\ni x \mapsto L(x)+B_F(x,a+e_0)\in {\Bbb F}_{2^n}$ are one-to-one mappings from $T_0$ to ${\Bbb F}_{2^n}$ for any $a\in T_0$.
(Hence, $G(x)=F(x)+\Tr(x)L(x)$ is a quadratic APN function on ${\Bbb F}_{2^n}$ if, and only if, $L_a(x)=0$ implies $x=0$ for any $a\in T_0$).
\end{theorem}
\begin{proof}
Let $f:=F\vert_{T_0}$ be the restriction of $F$ to $T_0$; $f$ is a quadratic APN function from $T_0$ to ${\Bbb F}_{2^n}$. For $x\in T_0$, we have $F(x)+\Tr(x)L(x)=f(x)$ and $F(x+e_0)+\Tr(x+e_0)L(x+e_0)=f(x)+L(x)+B_F(x,e_0)+L(e_0)+F(e_0)+F(0)$. Let $G'$ be a function on ${\Bbb F}_{2^n}$ defined by $G'(x):=f(x)$ and $G'(x+e_0):=f(x)+L(x)+B_F(e_0,x)$ for $x\in T_0$, then $G'(x)=F(x)+\Tr(x)(L(x)+L(e_0)+F(e_0)+F(0))$ for $x \in {\Bbb F}_{2^n}$, which is EA-equivalent to $G(x)=F(x)+\Tr(x)L(x)$. By Proposition~\ref{prop2}, $G'$ is an APN function if and only if  $T_0\ni x\mapsto L(x)+B_F(x,e_0)+B_F(x,a)\in {\Bbb F}_{2^n}$ are one-to-one mappings for any $a\in T_0$. Thus, $G(x)=F(x)+\Tr(x)L(x)$ is a quadratic APN function on ${\Bbb F}_{2^n}$ if and only if $L_a\colon T_0\ni x \mapsto L(x)+B_F(x,a+e_0)\in {\Bbb F}_{2^n}$ are one-to-one mappings from $T_0$ to ${\Bbb F}_{2^n}$ for any $a\in T_0$.
\end{proof}

\begin{example}	
	Let $e_0$ be a fixed element of $\mathbb{F}_{2^n}$ such that $\Tr(e_0) = 1$. For the quadratic function $F(x)=x^3$ on $\F_2^n$, we consider (for different values of $n$) linear mappings $L$ on $\mathbb{F}_{2^n}$ with $L(e_0)=0$ satisfying the conditions of Theorem~\ref{th1}. For $F(x) = x^3$ on $\mathbb {F}_{2^4}$, there are 448 such mappings $L$; for $F(x) = x^3$ on $\mathbb{F}_{2^5}$, there are $4\,608$ such mappings $L$; and for $F(x) = x^3$ on $\mathbb{F}_{2^6}$, we found about $40\,000$ such mappings $L$.
\end{example}	

As the previous example indicates, it is possible to obtain many APN functions by modifying the values of a given APN function on a hyperplane using an affine function. This observation essentially leads to the following notion of equivalence.

\begin{definition}\label{def: H-equivalence}
    Let $H$ be a hyperplane of $\F_2^n$. Let $F$ and $G$ be APN functions on $\F_2^n$. We say that functions $F$ and $G$ are \textit{$H$-equivalent} if there exists an affine function $L\colon\F_2^n\to\F_2^n$ s.t. for all $x\in\F_{2^n}$ it holds that $G(x)=F(x)$ if $x\in H$, and $G(x)=F(x)+L(x)$ if $x \notin H$.
\end{definition}

The following result demonstrates that $H$-equivalence can be used to derive not only numerous APN functions from a single one but also many that are EA-inequivalent.

\begin{theorem}\label{rem: apns in dim 6}
    Up to EA-equivalence, every quadratic APN function on $\F_{2^6}$ has the form $\F_{2^6}\ni x \mapsto x^3 + \Tr(x)L(x)$, where $L$ is a linearized polynomial on $\F_{2^6}$. 
\end{theorem}
\begin{proof}
First, we observe $H$-equivalence can be studied in the representatives of EA-equivalence classes. Let $G$ be an APN function on $\F_{2^n}$ defined by $G(x)=F(x)+1_{H}(x)L(x)$, for $x\in\F_2^n$, where $H$ is a hyperplane given by $H=\{x \in\F_2^n \colon a\cdot x=\beta\}$, for some $a\in\F_2^n$ and $\beta \in \F_2$. Consider the function $G'=A\circ G \circ B+C$ which is EA-equivalent to $G$, i.e., $A,B\colon\F_2^n\to\F_2^n$ are affine invertible mappings and $C\colon\F_2^n\to\F_2^n$ is affine. Denote by $1_H\colon\F_2^n\to\F_2$ the indicator function of $H\subset \F_2^n$, i.e., $1_H(x)=1$ if $x\in H$ and $1_H(x)=0$ if $x\notin H$. By definition of $G'$, we have 
\begin{equation*}
	\begin{split}
		G'(x)&=A(F(B(x))+C(x)+A(1_{H}(B(x))L(B(x))))\\
		     &=F'(x)+1_{H}(B(x))L'(x)\\
		     &=F'(x)+1_{H'}(x)L'(x),
	\end{split}
\end{equation*}
where $F'=A\circ F \circ B+C$, $L'=A\circ L \circ B$ and the hyperplane $H'$ is defined by $H'=\{x \in\F_2^n \colon a\cdot B(x)=\beta\}$. In this way, it is enough to show that a representative of every EA-equivalence class of quadratic APN functions on $\F_2^6$ can be obtained from $\F_{2^6}\ni x \mapsto x^3$ using $H$-equivalence.

\begin{table}[h!]
	\centering
    \caption{APN functions $G_i(x)=x^3+\Tr(x)L_i(x)$ on $\F_{2^6}$ and their EA-equivalence to Dillon's quadratic APN functions $D_k(x)$.}
	\scalebox{0.95}{$\begin{tabular}{|c|l|c|c|}
			\hline
			$G_i$ & \multicolumn{1}{|c|}{$L_i(x)$} &  $D_k$ \\ \hline
			$G_1$ & $L_1(x)=0$ & $D_1$ \\
			$G_2$ & $L_2(x)=\alpha^{42}x+\alpha^{3}x^2+\alpha^{34}x^{2^2}+\alpha^{59}x^{2^3}+\alpha^{59}x^{2^4}+\alpha^{12}x^{2^5}$      & $D_{10}$  \\
			$G_3$ & $L_3(x)=\alpha^{18}x+\alpha^{60}x^2+\alpha^{17}x^{2^2}+\alpha^{4}x^{2^3}+\alpha^{17}x^{2^4}+\alpha^{4}x^{2^5}$        & $D_2$  \\
			$G_4$ & $L_4(x)=\alpha^{18}x+\alpha^{60}x^2+\alpha^{57}x^{2^2}+\alpha^{7}x^{2^3}+\alpha^{32}x^{2^4}+\alpha^{62}x^{2^5}$      & $D_3$  \\
			$G_5$ & $L_5(x)=\alpha^{42}x+\alpha x^2+\alpha^{29}x^{2^2}+\alpha^{55}x^{2^3}+\alpha^{9}x^{2^4}+\alpha^{56}x^{2^5}$         & $D_5$    \\
			$G_6$ & $L_6(x)=\alpha^{42}x+\alpha^{21}x^2+\alpha^{4}x^{2^3}+\alpha^{48}x^{2^4}+\alpha^{16}x^{2^5}$                      & $D_6$     \\
			$G_7$ & $L_7(x)=\alpha^{42}x+\alpha^{19}x^2+\alpha^{51}x^{2^2}+\alpha^{59}x^{2^3}+\alpha^{26}x^{2^4}+\alpha^{38}x^{2^5}$   & $D_7$   \\
			$G_8$ & $L_8(x)=\alpha^{42}x+\alpha^{19}x^2+\alpha^{60}x^{2^2}+\alpha^{11}x^{2^3}+\alpha^{25}x^{2^4}+\alpha^{13}x^{2^5}$    & $D_{12}$  \\
			$G_9$ & $L_{9}(x)=\alpha^{42}x+\alpha^{21}x^2+\alpha^{22}x^{2^2}+\alpha^{31}x^{2^3}+\alpha^{15}x^{2^4}+\alpha^{61}x^{2^5}$ & $D_{13}$ \\
			$G_{10}$ & $L_{10}(x)=\alpha^{42}x+\alpha^{47}x^2+\alpha^{35}x^{2^2}+\alpha^{54}x^{2^3}+\alpha^{23}x^{2^4}+\alpha^{27}x^{2^5}$  & $D_8$ \\
			$G_{11}$ & $L_{11}(x)=\alpha^{42} x+\alpha^{21} x^2+ \alpha^{23} x^{2^2}+\alpha^{32} x^{2^3}+ \alpha^{14} x^{2^4}+ \alpha^{51} x^{2^5}$ & $D_{11}$ \\
			$G_{12}$ & $L_{12}(x)=\alpha^{42} x+\alpha^{21} x^2+ \alpha^4 x^{2^2}+\alpha^{56} x^{2^3}+ \alpha^{17} x^{2^4}+ \alpha^{20} x^{2^5}$  & $D_4$ \\
			$G_{13}$ & $L_{13}(x)=\alpha^{42} x+\alpha^{21} x^2+ \alpha^{27} x^{2^3}+ \alpha^{34} x^{2^4}+ \alpha^{52} x^{2^5}$ & $D_9$ \\ \hline
		\end{tabular}$}
	\label{table 1}
\end{table}
%
 Let the multiplicative group of $\F_{2^6}$ be defined by $(\F_{2^6})^*=\langle a \rangle$, where $\alpha^6 + \alpha^4 + \alpha^3 + \alpha + 1=0$. In Table~\ref{table 1}, we give the linearized polynomials $L_i\colon\F_{2^6}\to\F_{2^6}$ that define quadratic APN functions $G_i(x)=F(x)+\Tr(x)L_i(x)$. We also indicate, which of these functions are EA-equivalent to the $k$-th Dillon APN function, whose univariate representations can be found in~\cite[Table 4]{pott}. The EA-equivalence was verified using Magma~\cite{MAGMA:1997}.
\end{proof}

\begin{remark}
1. In~\cite{Edel_non_quadratic_APN:200959}, it was shown that switching cannot be used to obtain representatives of all EA-equivalence classes of quadratic APN functions starting from a single class. However, this can be achieved using $H$-equivalence with the hyperplane $H = \{x\in \F_{2^n} \colon \Tr(x) = 0\}$. Indeed, by Theorem~\ref{th1}, two EA-inequivalent quadratic APN functions on $\F_{2^6}$ have the form $F(x)=x^3+\Tr(x)L(x)$ and $F'(x)=x^3+\Tr(x)L'(x)$, from what follows that $F'(x)=F(x)+\Tr(x)L''(x)$, where $L'':=L+L'$ is a linear mapping on $\F_{2^6}$.

\noindent 2. For an $(n,n)$-function $F$, the mapping $W_F\colon\F_{2^n}\times\F_{2^n}\to\mathbb{Z}$ defined by $W_F(a,b)=\sum_{x\in {\Bbb F}_{2^n}}(-1)^{\Tr(bF(x)+ax)}$, where $a,b\in\F_{2^n},b\in \F_{2^n}^\times$, is called the \textit{Walsh transform} of $F$. The  multiset $\mathcal{W}_F=\{ * \ W_F(a,b)\colon a\in {\Bbb F}_{2^n}, b\in {\Bbb F}_{2^n}^\times  \ * \}$ is called the \textit{Walsh spectrum} of $F$. If $n$ is even, an APN function $F$ on $\F_{2^n}$ is said to have the \textit{the classical Walsh spectrum} if $$\begin{gathered}
\mathcal{W}_F=\left\{* \ 0\left[2^{n-2} \cdot\left(2^n-1\right)\right], \pm 2^{\frac{n+2}{2}}\left[\frac{1}{3}\left(2^n-1\right) \cdot\left(2^{n-3} \pm 2^{\frac{n-4}{2}}\right)\right],\right. \\
\left. \pm 2^{\frac{n}{2}}\left[\frac{2}{3}\left(2^n-1\right) \cdot\left(2^{n-1} \pm 2^{\frac{n}{2}-1}\right)\right] \ * \right\}.
\end{gathered}$$ 
Otherwise, $F$ is said to have a \textit{non-classical Walsh spectrum}; see \cite{pott}. Note that most of the known examples and constructions of quadratic APN functions on $\F_{2^n}$ have the classical Walsh spectrum and only a few sporadic examples with a non-classical spectrum are known; for the recent advances in such functions, we refer to~\cite{BeierleLP22}. Note that, up to EA-equivalence, the function $G_7$ in Table~\ref{table 1} is the only quadratic APN function on $\F_{2^6}$ 
with a non-classical Walsh spectrum since its Walsh transform takes the values in the set $\{0,\pm 8, \pm 16, \pm 32\}$. Remarkably, this function can be constructed via $H$-equivalence from the Gold APN function $\F_{2^6} \ni x \mapsto x^3$ with the classical Walsh spectrum. 
\end{remark}

\subsection{A necessary and sufficient condition for the APN-ness of $G(x)=x^3+\operatorname{Tr}(x)L(x)$}\label{subsec: 4.2}
In the previous subsection, we showed that one can obtain many EA-inequivalent APN functions by modifying the mapping $x\in\F_{2^n}\mapsto x^3$ on the ``trace-one'' hyperplane. In this regard, it is natural to ask: When does the mapping $G(x)=x^3+Tr(x)L(x)$ define an APN function $G$ on $\F_{2^n}$? In the following theorem, we provide an answer to this question using certain exponential sums.
\begin{theorem}\label{th: Kloosterman sums iff}
	Let $G$ be a function from $\mathbb{F}_{2^n}$ to $\mathbb{F}_{2^n}$ defined as $$G(x)=x^3+ \Tr(x)L(x).$$ Then, $G$ is APN if and only if
	\begin{equation}\label{eq: iff condition}
	    \sum_{x \in \mathbb{F}_{2^n}\setminus\{0,1\}}(-1)^{ \Tr\left(\frac{x^2L(x^2+x)}{(x^2+x)^3}\right) }-\frac{1}{2}\sum_{x \in \mathbb{F}_{2^n}\setminus\{0,1\}}(-1)^{ \Tr\left(\frac{L(x^2+x)}{(x^2+x)^3}\right) } =2^{n-1}-1.
	\end{equation}
\end{theorem}

\begin{proof}
    Assume that for every $a, b \in \mathbb{F}_{2^n}$ with $a \neq 0$, the equation $G(x + a) + G(x) = b$ has at most two solutions. Since $G(x)$ is a quadratic function, we just count the number of solutions of the equation $G(x + a) + G(x) + G(a) = 0$. This gives
	\begin{equation*}
		\begin{split}
			G(x+a)+G(x)+G(a) &= (x+a)^3+ \Tr(x+a)L(x+a)\\&+x^3+ \Tr (x)L(x)+a^3+ \Tr(a)L(a) = 0.
		\end{split}
	\end{equation*}
	After simplification, we have
	\begin{equation}\label{eq: 1.1} x^2a + a^2x +  \Tr(x)L(a) +  \Tr(a)L(x) = 0.
	\end{equation}	
 We have to show that this equation has only two solutions for $a\ne 0$. Note that we can see immediately two solutions, 
namely $x=0$ and $x=a$. 

   \noindent \textbf{I.} Consider the case $\Tr(a) = 0$. Then, Eq.~\eqref{eq: 1.1} becomes
	\begin{equation}\label{eq: 1.2}
	x^2a + a^2x +  \Tr(x)L(a) = 0.
	\end{equation}
 We show that this equation has exactly two solutions if and only if the condition in Eq.~\eqref{eq: iff condition} holds. 
	Assume that $ \Tr(x) = 0$. Then, Eq.~\eqref{eq: 1.2} becomes
	$$x^2a + a^2x = 0.$$
	The above equation has exactly two solutions $x\in\{0, a\}$. Now, we show that the case $ \Tr(x) = 1$ gives no solutions if and only if Eq.~\eqref{eq: iff condition} holds. Assume that $\Tr(x) = 1$. Then, Eq.~\eqref{eq: 1.2} becomes
	\begin{equation}\label{eq: 01}
	    x^2a + a^2x + L(a) = 0.
	\end{equation}
	Let $H = \{x\in\mathbb{F}_{2^n}\colon \Tr(x) = 0\}, \hat{H}=\mathbb{F}_{2^n}\backslash H$ and $H^\times=H\backslash\{0\}$. 
 Let
 \begin{equation*}
     \begin{split}
         t_0=&\#\{x\in H, a \in H^\times \colon x^2a + a^2x + L(a) = 0\},\mbox{ and}\\
	t_1=&\#\{x\in \hat{H}, a \in H^\times \colon x^2a + a^2x + L(a) = 0\}.
     \end{split}
 \end{equation*}
    Then, we define
	\begin{equation*}
	    \begin{split}
        s_0=&\sum_{\beta\in \mathbb{F}_{2^n}, a\in H^\times, x \in H }(-1)^{ \Tr(\beta(x^2a+a^2x+L(a)))}=2^nt_0, \\ 
         s_1=&\sum_{\beta\in \mathbb{F}_{2^n}, a\in H^\times, x \in \mathbb{F}_{2^n}}(-1)^{ \Tr(\beta(x^2a+a^2x+L(a)))}=2^n(t_0+t_1).
	    \end{split}
	\end{equation*}
	In this way, the value of $t_1$ is given by
	$$t_1=\frac{s_1-s_0}{2^n}.$$
 First, we consider the expressions of $s_1$ and $s_0$:
    \begin{equation}\label{eq: s1}
	\begin{split}
			s_1&=\sum_{\beta\in \mathbb{F}_{2^n}, a\in H^\times}(-1)^{ \Tr(\beta L(a))}\sum_{x\in \mathbb{F}_{2^n}}(-1)^{ \Tr(\beta(x^2a+a^2x))}\\
			&=\sum_{\beta\in \mathbb{F}_{2^n}, a\in H^\times}(-1)^{ \Tr(\beta L(a))}\sum_{x\in \mathbb{F}_{2^n}}(-1)^{ \Tr(\beta a^3((\frac{x}{a})^2+(\frac{x}{a})))},\\
        \end{split}
   \end{equation}
   and similarly
   \begin{equation}\label{eq: s0}
   \begin{split}
		s_0&=
  \sum_{\beta\in \mathbb{F}_{2^n}, a\in H^\times} (-1)^{ \Tr(\beta L(a))} \sum_{x \in H }(-1)^{ \Tr(\beta(x^2a+a^2x))}\\
			&=\sum_{\beta\in \mathbb{F}_{2^n}, a\in H^\times} (-1)^{ \Tr(\beta L(a))} \sum_{x \in H }(-1)^{ \Tr\left(\beta a^3 \left( (\frac{x}{a})^2 + (\frac{x}{a})\right)\right)}.
		\end{split}
	\end{equation}
Since both expressions depend on the evaluation of $\Tr(\beta a^3 ((\frac{x}{a})^2 + (\frac{x}{a})))$, let us consider the solutions of the equation $\Tr(\beta a^3 ((\frac{x}{a})^2 + (\frac{x}{a})))=0$.  Note that for $\beta\in\F_{2^n}$ and a fixed $a\in\F_{2^n}$ with $\Tr(a)=0$ we have that
	$$ \Tr\left(\beta a^3 \left( \left(\frac{x}{a}\right)^2 + \left(\frac{x}{a}\right)\right)\right)=0$$
	holds for all $x \in H$ if and only if 
	$$\Tr\left(x^2\left(\frac{\beta a^3}{a^2}+\frac{\beta^2a^6}{a^2}\right)\right) = 0\iff \left(\frac{\beta a^3}{a^2}+\frac{\beta^2a^6}{a^2}\right) = 0 \mbox{ or } 1.$$
	First, we consider
	$$\left(\frac{\beta a^3}{a^2}+\frac{\beta^2a^6}{a^2}\right) = 0.$$
	Then, we have the following cases:
	\begin{itemize}
		\item[(1)] $\beta = 0$,
		\item[(2)] $\beta a^3 + \beta^2a^6=0$ 
		if and only if $\beta = \frac{1}{a^3}$
  (provided $\beta \ne 0$).
        \end{itemize}
  Now, we consider the case when
		$$\left(\frac{\beta a^3}{a^2}+\frac{\beta^2a^6}{a^2}\right) = 1.$$
In this case, $\Tr\left(\beta a^3 \left( \left(\frac{x}{a}\right)^2 + \left(\frac{x}{a}\right)\right)\right)=0$ holds for all $x\in H$, but 
not for all $x\in\mathbb{F}_{2^n}$, hence this case occurs just
in Eq.~\eqref{eq: s0}, but not in Eq.~\eqref{eq: s1}. The latter corresponds to the following case
	\begin{itemize}
		\item[(3)] $\beta a^3 + \beta^2a^6 = a^2.$
	\end{itemize}
        Considering these cases, we can now compute the values of $s_1$ and $s_0$. With the first two cases, Eq.~\eqref{eq: s1} becomes 
        \begin{equation}\label{eq: s1a final}
		\begin{split}
			s_1&=\sum_{\beta\in \mathbb{F}_{2^n}, a\in H^\times}(-1)^{ \Tr(\beta L(a))}\sum_{x\in \mathbb{F}_{2^n}}(-1)^{ \Tr(\beta a^3((\frac{x}{a})^2+(\frac{x}{a})))}\\
			&=2^n(2^{n-1}-1)+2^n\left(\sum_{a\in H^\times}(-1)^{ \Tr\left(\frac{L(a)}{a^3}\right)}\right),
		\end{split}
	\end{equation}
        where the first term corresponds to case (1) $\beta = 0$, while the second term corresponds to case (2) $\beta a^3 = 1$. 
        Consider now Eq.~\eqref{eq: s0}. Let $\gamma=\beta a^3$. Then, we have $\beta=\frac{\gamma}{a^3}$. From case (3), we have that $\gamma+\gamma^2=a^2$. Therefore, we obtain
	\begin{equation*}
		\begin{split}
			s_0=&2^{n-1}(2^{n-1}-1)+2^{n-1}\sum_{a\in H^\times}(-1)^{ \Tr\left(\frac{L(a)}{a^3}\right)}\\+&2^{n-1}\left(\sum_{\substack{a\in H^\times s.t. \; a^2=\gamma+\gamma^2 \\ for \; \gamma \in \mathbb{F}_{2^n}\backslash\{0,1\}}}(-1)^{ \Tr\left(\gamma \frac{L(a)}{a^3}\right)}\right),
		\end{split}
	\end{equation*}
    where the first two terms are again obtained as in Eq.~\eqref{eq: s1a final}. Now, substitute $a = \sqrt{\gamma+\gamma^2}$, the above equation becomes
	\begin{equation}\label{eq: s1 final}
		\begin{split}
			s_0=&2^{n-1}(2^{n-1}-1)+2^{n-1}\sum_{a\in H^\times}(-1)^{ \Tr\left(\frac{L(a)}{a^3}\right)}\\+&2^{n-1}\left(\sum_{\gamma \in \mathbb{F}_{2^n}\setminus\{0,1\}}(-1)^{ \Tr\left( \frac{\gamma L(\sqrt{\gamma+\gamma^2})}{(\sqrt{\gamma+\gamma^2})^3}\right)}\right)\\
            =&\frac{s_1}{2} + 2^{n-1}\left(\sum_{\gamma \in \mathbb{F}_{2^n}\setminus\{0,1\}}(-1)^{ \Tr\left( \frac{\gamma L(\sqrt{\gamma+\gamma^2})}{(\sqrt{\gamma+\gamma^2})^3}\right)}\right).
		\end{split}
	\end{equation}
Eq.~\eqref{eq: 1.2} has exactly two solutions 
for $a\ne 0$ with $\Tr(a)=0$ if and only if $t_1=s_1-s_0=0$, hence $s_0=s_1$. Substituting $s_0=s_1$ in Eq.~\eqref{eq: s1 final} and replacing $\gamma$ with $x^2$, we get
    \begin{equation}\label{eq: s1 final value}
        s_1=2^n\sum_{x \in \mathbb{F}_{2^n}\setminus\{0,1\}}(-1)^{ \Tr\left(\frac{x^2L(x^2+x)}{(x^2+x)^3}\right) }.
    \end{equation}
    Substituting the obtained above value of $s_1$ in Eq.~\eqref{eq: s1a final}, 
    we get
    \begin{equation*}
        \sum_{x \in \mathbb{F}_{2^n}\setminus\{0,1\}}(-1)^{ \Tr\left(\frac{x^2L(x^2+x)}{(x^2+x)^3}\right) }-\sum_{a\in H^\times}(-1)^{ \Tr\left(\frac{L(a)}{a^3}\right)} =2^{n-1}-1,
    \end{equation*}
    which is equivalent to the original condition in Eq.~\eqref{eq: iff condition} since $$\sum_{a\in H^\times}(-1)^{ \Tr\left(\frac{L(a)}{a^3}\right)} = \frac{1}{2}\sum_{x \in \mathbb{F}_{2^n}\setminus\{0,1\}}(-1)^{ \Tr\left(\frac{L(x^2+x)}{(x^2+x)^3}\right) },$$ because $x^2+x$ runs through all the elements in  $H^\times$ twice if $x\in\mathbb{F}_{2^n}\setminus\{0,1\}$.  This completes the proof of the first case $\Tr(a) = 0$. \ \\

    \noindent\textbf{II.} Consider the case $\Tr(a) = 1$. Then, Eq.~\eqref{eq: 1.1} becomes
	\begin{equation}\label{eq: 1.2a}
	x^2a + a^2x +  \Tr(x)L(a) + L(x)= 0.
	\end{equation}
	Assume that $ \Tr(x) = 0$. Then, Eq.~\eqref{eq: 1.2a} becomes
	$$x^2a + a^2x + L(x) = 0,$$
    which is symmetric to Eq.~\eqref{eq: 01} w.r.t. a change of the variables. Assuming that there is a solution $x\ne 0$ with $\Tr(x)=0$, one can find $a\in\F_{2^n}$ s.t. it satisfies Eq.~\eqref{eq: 01}. 
    Case \textbf{I} shows that $\Tr(a)=0$, a contradiction.
    \\
    
    \noindent Assume now that $ \Tr(x) = 1$. 
    In this case,
$$x^2a+a^2x+L(a)+L(x) = 0$$
or, putting $y=x+a$,
\begin{equation}\label{eq: 2}
y^2a+a^2y+L(y)=0.\end{equation}
One solution is $y=0$, hence $x=a$. But for a given $y\ne 0$ with $\Tr(y)=0$, there is no $a$ with $\Tr(a)=1$ which satisfies Eq.~\eqref{eq: 2}, as in case \textbf{I}, hence also for $\Tr(a)=1$ we have the only two solutions
$x\in\{0,a\}$ for Eq.~\eqref{eq: 1.1}.

     In this way, we conclude that $G$ is APN if and only if the condition in Eq.~\eqref{eq: iff condition} holds.
\end{proof}
\begin{remark}
    One can check with a computer algebra system that linearized polynomials in Table~\ref{table 1} indeed satisfy the APN condition given in Theorem~\ref{th: Kloosterman sums iff}.
\end{remark}

\section{Modifying APN functions on affine subspaces of codimension two}\label{se:construction_method}\label{sec: 5}
In the previous section, we showed that it is possible to construct a new APN function from a given one by adding a linear function defined on a hyperplane. In this section, we show that new APN functions can be obtained by adding carefully selected constant functions on disjoint affine subspaces of codimension two. The following theorem provides a complete characterization of the APN property for such modifications.
\begin{theorem}\label{th:Subspaces}
	Let $F$ be an APN function from $\mathbb{F}_2^n$ to $\mathbb{F}_2^n$. Let $U$ be an $(n-2)$-dimensional subspace of $\mathbb{F}_2^n$ and let $U_1=U+u_1,\ U_2=U+u_2,\ U_3=U+u_3$ and $U_4=U+u_4$ be the four cosets of $U$ such that $\mathbb{F}_2^n=U_1 \cup U_2 \cup U_3 \cup U_4,$ where $u_1=0$ and $u_2,\ u_3,\ u_4 \in \mathbb{F}_2^n$. Define the function $G\colon\mathbb{F}_2^n\to \mathbb{F}_2^n$ as follows
	\begin{equation}\label{eq: definition of G on 4 cosests}
		G(x) = \left\lbrace
		\begin{array}{cccc}
			F(x)+a_1, & & \text{for }  x \in U_1,  \\[1ex]
			F(x)+a_2, & & \text{for }  x \in U_2,  \\[1ex]
			F(x)+a_3, & & \text{for }  x \in U_3,  \\[1ex]
			F(x)+a_4, & & \text{for }  x \in U_4,  
		\end{array}\right.
	\end{equation}
    where $a_1,a_2,a_3,a_4 \in \mathbb{F}_2^n$. Then, the function $G$ is APN if and only if the condition $$F(x_1)+F(x_2)+F(x_3)+F(x_4) \neq a_1+a_2+a_3+a_4$$ holds for all $2$-dimensional affine subspaces $\{x_1,x_2,x_3,x_4\}$ of\ $\mathbb{F}_2^n$  with $|\{x_1,x_2,x_3,x_4\}\cap U_i|=1$, for all $i=1,2,3,4$.  
\end{theorem}	
\begin{proof}
	Let $U$ be an $(n-2)$-dimensional linear subspace of $\mathbb{F}_2^n$. We decompose $\mathbb{F}_2^n$ into four cosets of $U$ such that $\mathbb{F}_2^n=U_1 \cup U_2 \cup U_3 \cup U_4,$ where $U_1=U+u_1=U,\ U_2=U+u_2,\\ U_3=U+u_3$ and $U_4=U+u_4$, where $u_1=0$ and $u_2, u_3, u_4 \in \mathbb{F}_2^n$. Since $F$ is an APN function, we have that the condition $$F(x_1)+F(x_2)+F(x_3)+F(x_4) \neq 0$$ holds for all $2$-dimensional affine subspaces $\{x_1,x_2,x_3,x_4\}$ of $\mathbb{F}_2^n$. Now, we consider the distribution of affine $2$-dimensional subspaces $\{x_1,x_2,x_3,x_4\}$ in the subspaces $U_1,U_2,U_3,U_4$. 
	
	Recall that the set $\{x_1,x_2,x_3,x_4\}$ is an affine $2$-dimensional subspace of $\mathbb{F}_2^n$ if $x_1,x_2,x_3,x_4$ are pairwise different and $x_1+x_2+x_3+x_4=0$. Assume that $x_1,x_2 \in U_i$, $x_3 \in U_j$ and $x_4 \in U_k$, where $i,j,k$ are pairwise different. Then, $x_1+x_2 \in U_1$ and $x_3+x_4  \notin U_1$. Therefore, $x_1+x_2+x_3+x_4=0 \in U_1$ is not possible. The only possible distributions of $x_1,x_2,x_3,x_4$ between $U_i$ are either:
	\begin{itemize}
		\item[1)] $x_1,x_2,x_3,x_4 \in U_i$, for some $i$,
		\item[2)] $x_1,x_2 \in U_i$ and $x_3,x_4 \in U_j$, for some $i\neq j$,
		\item[3)] $x_1\in U_i,x_2\in U_j,x_3\in U_k,x_4\in U_l$, for $i\neq j\neq k \neq l$ (equivalently, the condition $| \{x_1,x_2,x_3,x_4\} \cap U_i| =1$ holds for all $i=1,2,3,4$).
	\end{itemize}
	In order to derive a necessary and sufficient condition for the APN-ness of the function $G$ defined in Eq.~\eqref{eq: definition of G on 4 cosests}, we need to verify that $G(x_1)+ G(x_2)+G(x_3)+G(x_4)\neq 0$ holds for all 2-dimensional flats $\{x_1,x_2,x_3,x_4\}$. To do so, we consider three possible distributions of the elements $x_i$ between the subspaces $U_j$.
	
	In the first case, we have that $x_1,x_2,x_3,x_4 \in U_i$, for some $i$. Then, 
	\begin{equation*}
		\begin{split}
			&G(x_1)+G(x_2)+G(x_3)+G(x_4)\\
			=&(F(x_1)+a_i)+(F(x_2)+a_i)+(F(x_3)+a_i)+(F(x_4)+a_i)\neq 0
		\end{split}
	\end{equation*}
due to the APN property of $F$.
    
    In the second case, we have that $x_1,x_2 \in U_i$ and $x_3,x_4 \in U_j$, for some $i\neq j$. Consequently,  
    \begin{equation*}
    	\begin{split}
    		&G(x_1)+G(x_2)+G(x_3)+G(x_4)\\
    		=&(F(x_1)+a_i)+(F(x_2)+a_i)+(F(x_3)+a_j)+(F(x_4)+a_j)\neq 0
    	\end{split}
    \end{equation*}
again due to the APN property of $F$.

In the third case, we have that the condition $|\{x_1,x_2,x_3,x_4\} \cap U_i| =1$ holds for all $i=1,2,3,4$. W.l.o.g, we assume that $x_i\in U_i$. Then, we have
	\begin{equation*}
	\begin{split}
		&G(x_1)+G(x_2)+G(x_3)+G(x_4)\\
		=&(F(x_1)+a_1)+(F(x_2)+a_2)+(F(x_3)+a_3)+(F(x_4)+a_4)\\
		=& F(x_1)+F(x_2)+F(x_3)+F(x_4)+(a_1+a_2+a_3+a_4).
	\end{split}
\end{equation*}
In this way, the function $G$ is APN if and only if the condition $$F(x_1)+F(x_2)+F(x_3)+F(x_4) \neq a_1+a_2+a_3+a_4$$ holds for all 2-dimensional affine subspaces $\{x_1,x_2,x_3,x_4\}$ of $\mathbb{F}_2^n$ with the property $|\{x_1,x_2,x_3,x_4\}\cap U_i|=1$ for all $i=1,2,3,4$.
\end{proof}	

In the following example, we indicate that using Theorem~\ref{th:Subspaces} one construct an APN function, that is EA-inequivalent to a given one.

\begin{example}\label{ex:subspace_n_8_1}
	Let $F$ be an APN function from $\mathbb{F}_{2^8}$ to $\mathbb{F}_{2^8}$ defined as $F(x)=x^3$.  Let the multiplicative group of $\F_{2^8}$ be defined by $(\F_{2^8})^*=\langle a \rangle$, where $\alpha^8 + \alpha^4 + \alpha^3 + \alpha^2 + 1=0$.  First, we decompose $\mathbb{F}_{2^8}$ into four cosets $U_1, U_2, U_3, U_4$ defined by 
   \begin{equation*}
     \begin{split}
         U_1=&\{x \in \mathbb{F}_{2^8}\colon   \Tr^8_2(x)=0\},\\
         U_2=&\{x \in \mathbb{F}_{2^8}\colon \Tr^8_2(x)=1\},\\
         U_3=&\{x \in \mathbb{F}_{2^8}\colon \Tr^8_2(x)=\beta \},\\
         U_4=&\{x \in \mathbb{F}_{2^8}\colon \Tr^8_2(x)=\beta^2\},\\
     \end{split}
    \end{equation*}
    where $\beta$ is a root of $x^2+x+1$; in the defined above field $\F_{2^8}$, we have that $\beta\in \{\alpha^{85}, \alpha^{170}\}$. Compute the set 
    \begin{equation*}
    	\begin{split}
    		A&=\F_{2^8}\setminus \left\{\sum_{i=1}^4 F(x_i)\colon \sum_{i=1}^4 x_i=0 \mbox{ and  } x_i \in U_i,\mbox{ for } i=1,2,3,4 \right\} \\
    		&=\{0, 1, \alpha^{85}, \alpha^{170}\}.
    	\end{split}
    \end{equation*}
    We can choose $a_1,a_2,a_3,a_4 \in \mathbb{F}_{2^8}$ such that $a_1+a_2+a_3+a_4\in A$. Setting  $a_1=a_2=0,a_3=\alpha^{85}\beta,a_4=\alpha^{85}\beta^2$, we have that $a_1+a_2+a_3+a_4=\alpha^{85}\in A$. Define the function $G\colon\F_{2^8}\to\F_{2^8}$ as in Eq.~\eqref{eq: definition of G on 4 cosests}, that is
    $$G(x) = \left\lbrace
	\begin{array}{cccc}
		F(x)+a_1, & & \text{for }  x \in U_1,  \\[1ex]
		F(x)+a_2, & & \text{for }  x \in U_2,  \\[1ex]
		F(x)+a_3, & & \text{for }  x \in U_3,  \\[1ex]
		F(x)+a_4, & & \text{for }  x \in U_4,  
	\end{array}\right.
=\left\lbrace
\begin{array}{cccc}
	F(x), & & \text{if }   \Tr^8_2(x)=0,  \\[1ex]
	F(x), & & \text{if }   \Tr^8_2(x)=1,  \\[1ex]
	F(x)+\alpha^{85}\beta, & & \text{if }   \Tr^8_2(x)=\beta,  \\[1ex]
	F(x)+\alpha^{85}\beta^2, & & \text{if }   \Tr^8_2(x)=\beta^2.  
\end{array}\right.
	$$
	Then, the function $G$ is APN by Theorem~\ref{th:Subspaces}. For the constructed function $G$, we give a ``simplified'' finite filed representation. To do so, we recall that the polynomial representation of the trace function from $\F_{2^n}$ to its subfield $\F_{2^m}$ is given by $\Tr^n_m(x) =\sum_{i=0}^{\frac{n}{m}-1} x^{2^{i \cdot m}}$. First, we observe that  $\Tr^8_1(x)=\Tr^2_1(\Tr^8_2(x))=\Tr^8_2(x)+(\Tr^8_2(x))^2$. In this way, we have that
	if $\Tr^8_2(x)\in\{0,1\}$ then $\Tr^8_1(x)=0$. On the other hand, if $\Tr^8_2(x)\in\{\beta,\beta^2\}$ then $\Tr^8_1(x)=1$. In this way, the APN function $G\colon\mathbb{F}_{2^8}\to\mathbb{F}_{2^8}$ can be written as $$G(x)=x^3+\alpha^{85}\Tr^8_1(x)\Tr^8_2(x).$$ Finally, we notice that $G$ is EA-inequivalent to $F(x)=x^3$ on $\F_{2^8}$, since $\Gamma\mbox{-rank}(G)=13\,842\neq\Gamma\mbox{-rank}(F)=11\,818$.
\end{example}

\section{Conclusion and open problems}\label{sec: conclusion}
In this paper, we consider further secondary constructions of APN functions. The following questions related to the computational results obtained in Section~\ref{sec: 4} deserve further theoretical investigations:
\begin{enumerate}
	\item  As we show in Remark~\ref{rem: apns in dim 6}, every quadratic APN function in dimension 6 can be written (up to EA-equivalence) as a univariate polynomial $G_i(x)=x^3+\Tr(x)L_i(x)$, for $x\in\F_{2^6}$, where  $L_i\colon\F_{2^6}\to\F_{2^6}$ is a suitable linear mapping. Alternatively, as shown in~\cite{BCCCV_Isotopic_2020}, each quadratic APN function in dimension 6 can also be characterized (up to EA-equivalence) through the isotopic shift construction. We believe it would be valuable to explore these two approaches further and gain a deeper understanding of their connections.
	\item As we indicate in Remark~\ref{rem: apns in dim 6}, it is possible to construct APN functions with the non-classical spectrum from those with the classical one by modifying them on hyperplanes using suitable linear functions. Find more such examples in higher dimensions and analyze whether infinite families of APN functions with the non-classical Walsh spectrum of the form $\F_{2^n}\ni x \mapsto x^3+\Tr(x)L(x)$ exist.
	\item More general, find linearized polynomials satisfying the APN condition given in Eq.~\eqref{eq: iff condition}, which inevitably lead to quadratic APN functions of the form $\F_{2^n} \ni x 
 \mapsto x^3+\Tr(x)L(x)$. Is it possible to obtain APN functions of the form $\F_{2^n} \ni x  \mapsto x^3+\Tr(x)Q(x)$, where $Q$ is quadratic, or more general, nonlinear? 
\end{enumerate}
\section*{Acknowledgments} 
\noindent The authors express their gratitude to Faruk G\"olo\u{g}lu for providing the initial version of the proof for Theorem~\ref{th: Kloosterman sums iff} and for granting permission to develop these ideas further and publish it. In fact, Section~\ref{subsec: 4.2} is based on several discussions with him concerning the APN-ness of the functions of the form $F(x)+\Tr(x)L(x)$. Parts of this work were presented in ``The 8th International Workshop on Boolean Functions and Their Applications, BFA 2023'' by the first author~\cite{Taniguchi_BFA23}, and in the PhD thesis~\cite{Arshad_PhD_Thesis} of the last author. \ \\

\noindent Hiroaki Taniguchi is partly supported by JSPS KAKENHI Grant No. JP21K03343. Part of the research of the third author has been funded by the Deutsche Forschungsgemeinschaft (DFG, German Research Foundation)~--- Project No. 541511634.





\end{document}